\documentclass[11pt]{article}
\usepackage{amssymb}
\usepackage{color}

\usepackage{soul}

\newenvironment{proof}[1][Proof]{\textbf{#1.} }{\
\rule{0.5em}{0.5em}}

\usepackage{mathtools}
\newcommand{\bN}{\mathbf{N}}
\newcommand{\bK}{\mathbf{K}}
\newcommand{\ord}{\mbox{\rm ord}}
\newcommand{\ini}{\mbox{\rm in}}
\newcommand{\ch}{\mathrm{char}\,}
\newcommand{\mo}{\mathrm{mod}\,}

\newtheorem{defi}{Definition}[section]
\newtheorem{nota}[defi]{Remark}
\newtheorem{ejemplo}[defi]{Example}

\newtheorem{teorema}[defi]{Theorem}
\newtheorem{prop}[defi]{Proposition}
\newtheorem{lema}[defi]{Lemma}
\newtheorem{coro}[defi]{Corollary}

\begin{document}
\title{On the Milnor formula in arbitrary characteristic
\footnotetext{
     \noindent   \begin{minipage}[t]{4in}
       {\small
       2010 {\it Mathematics Subject Classification:\/} Primary 14H20;
       Secondary 32S05.\\
       Key words and phrases: plane singularity, Milnor number, degree of conductor, factorization of polar curve.\\
       The first-named author was partially supported by the Spanish Project
       MTM 2016-80659-P.}
       \end{minipage}}}

\author{Evelia R.\ Garc\'{\i}a Barroso and Arkadiusz P\l oski}

\maketitle
\hfill{Dedicated to Antonio Campillo on the occasion of his 65th birthday}
\vspace{1cm}

\begin{abstract} The Milnor formula $\mu=2\delta-r+1$ relates the Milnor number $\mu$, the double point number $\delta$ and the number $r$ of branches of a plane curve singularity. It holds over the fields of characteristic zero. Melle and Wall based on a result by Deligne proved the inequality $\mu\geq 2\delta-r+1$ in arbitrary characteristic and showed that the equality $\mu=2\delta-r+1$ characterizes the singularities with no wild vanishing cycles.  In this note we give an account of results on the Milnor formula in characteristic $p$. It holds if the plane singularity is Newton non-degenerate (Boubakri et al. Rev. Mat. Complut. (2010) 25) or if $p$ is greater than the intersection number of the singularity with its generic polar (Nguyen H.D., Annales de l'Institut Fourier, Tome 66 (5) (2016)).  Then we improve our result on the Milnor number of irreducible singularities (Bull. London Math. Soc. 48 (2016)). Our considerations are based on the properties of polars of plane singularities in characteristic $p$.
\end{abstract}

\section*{Introduction}
\label{intro}

\noindent John Milnor proved in his celebrated book \cite{Milnor} the formula 

\begin{equation}
\label{Milnor-formula}
\mu=2\delta-r+1, \tag{M}
\end{equation}

\noindent where $\mu$ is the  Milnor number $\mu$, $\delta$ the double point number  and $r$ the number  of branches of a plane curve singularity.
The Milnor's proof of  (\ref{Milnor-formula}) is based on topological considerations. A proof given by Risler \cite{Risler} is algebraic and shows that (\ref{Milnor-formula}) holds in characteristic zero.

\medskip

\noindent On the other hand Melle and Wall based on a resultd by Deligne \cite{Deligne} proved the inequality 
 $\mu\geq 2\delta-r+1$ in arbitrary characteristic and showed that the Milnor formula holds if and only if the  singularity has not wild {\em vanishing cycles} \cite{Melle}. In the sequel we  will call a {\em tame singularity} any plane curve singularity  verifying  (\ref{Milnor-formula}).

\medskip

\noindent Recently some papers on the singularities satisfying 
(\ref{Milnor-formula}) in characteristic $p$ appeared. In \cite{BGreuel}) the authors showed that 
planar Newton non-degenerate singularities are tame. Different notions of non-degeneracy for plane curve singularities are discussed in
\cite{G-N}. In \cite{Nguyen} the author proved that if the characteristic $p$  is greater than the kappa invariant then the singularity is tame. In \cite{GB-P2016} and \cite{Hefez} the case of irreducible singularities is investigated. Our aim is to give an account of the above-mentioned results. 

\medskip

\noindent In Section 1 we prove that any semi-quasihomogeneous singularity is tame. Our proof is different from that given in \cite{BGreuel} and can be extended to the case of Kouchnirenko nondegenerate singularities (\cite[Theorem 9]{BGreuel}). In Section 2 and 3 we generalize Teissier's lemma (\cite[Chap. II, Proposition 1.2]{Teissier}) relating the intersection number of the singularity with its polar and the Minor number to the case of arbitrary characteristic and reprove the result due to H.D. Nguyen \cite[Corollary 3.2]{Nguyen} in the following form: if $p>\mu(f)+\ord(f)-1$ then the  singularity is tame.

\medskip

\noindent Section 4 is devoted to the strengthened version of our result on the Milnor number of irreducible singularities.

\section{Semi-quasihomogeneous singularities}

\noindent  Let $\bK$ be an algebraically closed field of characteristic $p\geq 0$. For any formal power series $f\in \bK[[x,y]]$ we denote by $\ord(f)$ (resp. $\ini(f)$) the \emph {order} (resp. the \emph {initial form} of $f$). A power series $l\in \bK[[x,y]]$ is called a \emph{regular parameter} if $\ord(l)=1$. A  \emph{plane curve singularity}  (in short: {\em a singularity}) is a nonzero  power series $f$ of order greater than one. For any power series  $f,g\in \bK[[x,y]]$ we put $i_0(f,g):=\dim_{\bK} \bK[[x,y]]/(f,g)$ and called it the \emph{intersection number} of $f$ and $g$. The \emph{Milnor number} of $f$ is 
\[
\mu(f):=\dim _{\bK} \bK[[x,y]]/\left(\frac{\partial f}{\partial x}, \frac{\partial f}{\partial y}\right).
\] 

\noindent If $\Phi$ is an automorphism of $\bK[[x,y]]$ then $\mu(f)=\mu(\Phi(f))$ (see \cite[p. 62]{BGreuel}). If the characteristic of $\bK$ is $p=\ch \bK>0$ then we can have $\mu(f)=+\infty$ and $\mu(uf)<+\infty$ for a unit $u\in \bK[[x,y]]$ (take $f=x^p+y^{p-1}$ and $u=1+x$).

\medskip

\noindent  Let $f\in \bK[[x,y]]$ be a reduced (without multiple factors)  power series and
consider a regular parameter $l\in \bK[[x,y]]$. Assume that $l$ does not divide $f$. We call the \emph{polar  of $f$ with respect to $l$}  the power series

\[
{\cal P}_l(f)=\frac{\partial (f,l)}{\partial (x,y)}=\frac{\partial f}{\partial x}\frac{\partial l}{\partial y}-\frac{\partial f}{\partial y}\frac{\partial l}{\partial x}.
\]

\noindent If $l=-bx+ay$ for  $(a,b)\neq (0,0)$ then ${\cal P}_l(f)=a\frac{\partial f}{\partial x}+b\frac{\partial f}{\partial y}$.

\medskip

\noindent For any reduced power series $f$ we put  ${\cal O}_f=\bK[[x,y]]/(f)$, $ \overline{{\cal O}_f}$  the integral closure of  ${\cal O}_f$ in the total quotient ring of ${\cal O}_f$ and $\delta(f)=\dim_{\bK} \overline{{\cal O}_f}/{\cal O}_f$ (the double point number). Let ${\cal C}$ be the \emph{conductor} of ${\cal O}_f$, that is the largest ideal in ${\cal O}_f$  which remains an ideal in $ \overline{{\cal O}_f}$. We define $c(f)=\dim_{\bK} \overline{\cal O}_f/{\cal C}$ (the {\em degree of conductor}) and $r(f)$  the number of irreducible factors of $f$. The {\em semigroup} $\Gamma(f)$ associated with the irreducible power series $f$ is defined as the set of intersection numbers $i_0(f,h)$, where $h$ runs over power series such that $h\not\equiv 0$ ($\mo f$).

\medskip

\noindent The degree of conductor $c(f)$ is equal to the smallest element $c$ of $\Gamma (f)$ such that $c+N\in \Gamma (f)$ for all integers $N\geq 0$ (see \cite{Campillo-libro}, \cite{G-L-S}).

\medskip

\noindent For any reduced power series $f$ we define 
\[
\overline{\mu}(f):=c(f)-r(f)+1.
\]

\noindent In particular, if $f$ is irreducible then $\overline{\mu}(f)=c(f)$.
\begin{prop}
\label{pp:222}
Let $f=f_1\cdots f_r \in \bK[[x,y]]$ be a reduced power series, where $f_i$ is irreducible for $i=1,\ldots, r$. Then

\begin{enumerate}
\item[(i)] $\overline{\mu}(f)=\overline{\mu}(uf)$ for any unit $u$ of $\bK[[x,y]]$. 
\item [(ii)] 
\[
\overline{\mu}(f)+r-1=\sum_{i=1}^r\overline{\mu}(f_i)+2\sum_{1\leq i<j\leq r}
i_0(f_i,f_j).\]
\item [(iii)] Let $l$ be a regular parameter such that $i_0(f_i,l)\not \equiv 0$ \hbox{\rm(}$\mo p$\hbox{\rm)} for $i=1,\ldots, r$. Then 
\[
i_0(f,{\cal P}_l(f))=\overline{\mu}(f)+i_0(f,l)-1.
\]
\item [(iv)] $\overline{\mu}(f)=\mu(f)$ if and only if $\mu(f)=2\delta(f)-r(f)+1$.
\item [(v)]  $\overline{\mu}(f)\geq0$ and $\overline{\mu}(f)=0$ if and only if $\ord(f)=1$.

\end{enumerate}
\end{prop}

\noindent \begin{proof} Property $(i)$ is obvious. To check $(ii)$ observe that
\[
\sum_{i=1}^r\overline{\mu}(f_i)+2\sum_{1\leq i<j\leq r}
i_0(f_i,f_j)=\sum_{i=1}^r c(f_i)+2\sum_{1\leq i<j\leq r}
i_0(f_i,f_j)=c(f)=\overline{\mu}(f)+r-1,
\]

\noindent by \cite[Lemma 2.1, p. 381]{Campillo}.  Property $(iii)$ in the case $r=1$ reduces to the Dedekind formula $i_0(f,{\cal P}_l(f))=c(f)+i_0(f,l)-1$ provided that $i_0(f,l)\not \equiv 0$ ($\mo p$)  \cite[Lemma 3.1]{GB-P2016}. To check the general case we apply the Dedekind formula to the irreducible factors $f_i$ of $f$ and we get

\begin{eqnarray*}
i_0(f,{\cal P}_l(f))&=&\sum_{i=1}^r i_0(f_i,{\cal P}_l(f))=\sum_{i=1}^r i_0\left(f_i,{\cal P}_l(f_i)\frac{f}{f_i}\right)\\
&=&\sum_{i=1}^r\left(i_0(f_i,{\cal P}_l(f_i))+ \sum_{j\neq i}i_0(f_i,f_j)\right )\\
&=&\sum_{i=1}^r\left(\overline{\mu}(f_i)
+i_0(f_i,l)-1+\sum_{j\neq i}i_0(f_i,f_j)\right)\\
&=&\sum_{i=1}^r \overline{\mu}(f_i)+2 \sum_{1\leq i<j\leq r}
i_0(f_i,f_j)+i_0(f,l)-r\\
&=&\overline{\mu}(f)+r-1+i_0(f,l)-r=\overline{\mu}(f)+i_0(f,l)-1.
\end{eqnarray*}

\medskip
\noindent Property $(iv)$ follows since $c(f)=2\delta(f)$ for any reduced power series $f$ by the Gorenstein theorem (see for example \cite[Section 5]{Ploski2014}). 

\noindent Now we prove Property $(v)$.  If $f$ is irreducible then $\overline{\mu}(f)=c(f)\geq 0$ with equality if and only if $\ord(f)=1$. Suppose that $r>1$. Then by $(ii)$ we get 
\[
\overline{\mu}(f)+r-1\geq 2\sum_{1\leq i<j\leq r}
i_0(f_i,f_j)\geq r(r-1)
\]

\noindent  and $\overline{\mu}(f)\geq(r-1)^2>0$, which proves $(v)$.
\end{proof}

\medskip

\begin{nota}
\label{nota1}
Using Proposition \ref{pp:222} (ii) we check the following property:\\ Let $f=g_1\cdots g_s\in \bK[[x,y]]$ be a reduced  power series, where the power series $g_i$ for $i=1,\ldots, s$ are pairwise coprime. Then 
\[\overline{\mu}(f)+s-1=\sum_{i=1}^s \overline{\mu}(g_i)+2 \sum_{1\leq i<j\leq s}i_0(g_i,g_j).\]
\end{nota}

\noindent Let $\overrightarrow w=(n,m)\in (\bN_+)^2$  be a pair of strictly positive integers. In the sequel we call $\overrightarrow w$ a {\em weight}.

\medskip

\noindent Let $f=\sum c_{\alpha \beta}x^{\alpha}y^{\beta}\in \bK[[x,y]]$ be a power series. Then
\begin{itemize} 
\item the $\overrightarrow w$ {\em -order} of $f$ is $\ord_{\overrightarrow w}(f)=\inf \{\alpha n +\beta m\;:\; c_{\alpha \beta}\neq 0\}$,
\item  the $\overrightarrow w${\em -initial form} of $f$ is  $\ini_{\overrightarrow w}(f)=\sum_{\alpha n+\beta m=w} c_{\alpha \beta}x^{\alpha}y^{\beta},$ where $w=\ord_{\overrightarrow w}(f)$, 
\item $R_{\overrightarrow w}(f)=f-\ini_{\overrightarrow w}(f)$.
\end{itemize}

\noindent Thus $R_{\overrightarrow w}(f)$ is a power series of  $\overrightarrow w$-order greater than 
$\ord_{\overrightarrow w}(f)$.\\ Note that $\ord_{\overrightarrow w}(x)=n$ and $\ord_{\overrightarrow w}(y)=m$.

\medskip

\noindent A power series $f$ is {\em semi-quasihomogeneous} (with respect to $\overrightarrow w$) if the system of equations

\[\left\{ \begin{array}{l}
\frac{\partial}{\partial x} \ini_{\overrightarrow w}(f)=0,\\
\\
 \frac{\partial }{\partial y}\ini_{\overrightarrow w}(f)=0 
 \end{array}
 \right.
 \]
 
\noindent has the only solution $(x,y)=(0,0)$.

\medskip

\noindent A power series $f$ is {\em convenient} if $f(x,0)\cdot f(0,y)\neq 0$. 

\medskip

\noindent Suppose that $\ini_{\overrightarrow w}(f)$ is convenient and the line $\alpha n+\beta m= \ord_{\overrightarrow w}(f)$ intersects the axes in points $(m,0)$ and $(0,n)$. Let $d=\gcd(m,n)$. Then $\ini_{\overrightarrow w}(f)=F(x^{m/d},y^{n/d}),$ where $F(u,v)\in \bK[u,v]$ is a homogeneous polynomial of degree $d$.

\begin{prop}
\label{qh}
Suppose that $\ini_{\overrightarrow w}(f)$ has no multiple factors. Then 
\[\overline{\mu}(f)=\left(\frac{\ord_{\overrightarrow w}(f)}{n}-1\right)\cdot\left(\frac{\ord_{\overrightarrow w}(f)}{m}-1\right). \]
\end{prop}
\begin{proof} In the proof we will use lemmas collected in the Appendix. \\Observe that if $\ini_{\overrightarrow w}(f)$ has no multiple factors then  $\ini_{\overrightarrow w}(f)= m_{\overrightarrow w} (f) \left( \ini_{\overrightarrow w}(f)\right)^o$, where
$m_{\overrightarrow w} (f) \in \{1,x,y,xy\}$ and $\left( \ini_{\overrightarrow w}(f)\right)^o$ is a convenient power series or a constant. To prove the proposition we will use Hensel's Lemma (see Lemma \ref{Hensel}) and Remark \ref{nota1}. We have to consider several cases.\\

\noindent {Case 1:} $\ini_{\overrightarrow w}(f)=(\hbox{\rm const})\cdot x$ or $\ini_{\overrightarrow w}(f)=(\hbox{\rm const})\cdot y$.\\

\noindent In this case $\ord(f)=1$ and by Proposition \ref{pp:222}\emph{(v)} $\overline{\mu}(f)=0$. If $\ini_{\overrightarrow w}(f)=(\hbox{\rm const})\cdot x$ (resp. $\ini_{\overrightarrow w}(f)=(\hbox{\rm const})\cdot y$) then
$\ord_{\overrightarrow w}(f)=n$ (resp. $\ord_{\overrightarrow w}(f)=m$) and $$\left(\frac{\ord_{\overrightarrow w}(f)}{n}-1\right)\left(\frac{\ord_{\overrightarrow w}(f)}{m}-1\right)=0.$$\\

\noindent {Case 2:} $\ini_{\overrightarrow w}(f)=(\hbox{\rm const})\cdot xy.$\\

\noindent By Hensel's Lemma (see Lemma \ref{Hensel}) $f=f_1f_2$, where $\ini_{\overrightarrow w}(f_1)=c_1x$, 
$\ini_{\overrightarrow w}(f_2)=c_2y$ with constants $c_1,c_2\neq 0$. Using Remark \ref{nota1} and Lemma \ref{A1} we get
\[
\overline{\mu}(f)+1=\overline{\mu}(f_1f_2)+1=\overline{\mu}(f_1)+\overline{\mu}(f_2)+2i_0(f_1,f_2)=0+0+2.1
\]

\noindent  and $\overline{\mu}(f)=1$. On the other hand $\ord_{\overrightarrow w}(f)=n+m$ and 
\[
\left(\frac{\ord_{\overrightarrow w}(f)}{n}-1\right)\left(\frac{\ord_{\overrightarrow w}(f)}{m}-1\right)=1.
\]

\noindent {Case 3:} 
The power series $\ini_{\overrightarrow w} f$ is convenient. \\

\noindent Assume additionally that the line $n\alpha +m\beta=\ord_{\overrightarrow w}(f)$ intersects the axes in points $(m,0)$ and $(0,n)$. Let $d=\gcd(n,m)$.  Then the $\overrightarrow w$-initial form of $f$ is $$\ini_{\overrightarrow w} f=\prod_{i=1}^d\left(a_ix^{m/d}+b_iy^{n/d}\right),$$ where $a_ix^{m/d}+b_iy^{n/d}$ are pairwise coprime. By Hensel's Lemma (see Lemma \ref{Hensel}) we get a factorization $f=\prod_{i=1}^d f_i$, where $\ini_{\overrightarrow w} f_i=a_ix^{m/d}+b_iy^{n/d}$ for $i=1,\ldots, d$. The factors $f_i$ are irreducible with semigroup  $\Gamma(f_i)=\frac{m}{d}\bN+\frac{n}{d}\bN$ and 
$$\overline{\mu}(f_i)=c(f_i)=\left(\frac{m}{d}-1\right)\left(\frac{n}{d}-1\right)$$ (see, for example \cite{GB-P2015}). Moreover by Lemma \ref{A1} we have
\[
i_0(f_i,f_j)=\frac{\ord_{\overrightarrow w} f_i \ord_{\overrightarrow w} f_j}{mn}=\frac{mn}{d^2}, \;\hbox{\rm for } i\neq j
\]

\noindent and we get by Proposition \ref{pp:222}\emph{(ii)}
\[
 \overline{\mu}(f)+d-1=\sum_{i=1}^d  \overline{\mu}(f_i) + 2 \sum_{1\leq i<j\leq d}i_0(f_i,f_j)=d\left(\frac{m}{d}-1\right)\left(\frac{n}{d}-1\right)+2\frac{d(d-1)}{2}\frac{mn}{d^2},
\]

\noindent which implies $ \overline{\mu}(f)=(m-1)(n-1)=\left(\frac{\ord_{\overrightarrow w} f}{n}-1\right)\left(\frac{\ord_{\overrightarrow w} f}{m}-1\right)$, since the weighted order of $f$ is $\ord_{\overrightarrow w} f=mn$.\\

\medskip

\noindent Now consider the general case, that is when the line $n\alpha +m\beta=\ord_{\overrightarrow w}(f)$ intersects the axes in points $(m_1,0)=\left(\frac{\ord_{\overrightarrow w} f}{n},0\right)$   and $(0,n_1)=\left(0,\frac{\ord_{\overrightarrow w}(f)}{m}\right)$. Then $f$ is semi-quasihomogeneous with respect to $\overrightarrow{w_1}=(n_1,m_1)$ and the line $n_1\alpha +m_1\beta=\ord_{\overrightarrow{w_1}}(f)$ intersects the axes in points $(m_1,0)$ and $(0,n_1)$. By the first part of the proof we get 
\[
\overline{\mu}(f)=(m_1-1)(n_1-1)=\left(\frac{\ord_{\overrightarrow w}(f)}{n}-1\right)\left(\frac{\ord_{\overrightarrow w}(f)}{m}-1\right),
\]

\noindent which proves the proposition in Case 3.

\medskip

\noindent {Case 4:} $\ini_{\overrightarrow w}(f)=x\left(\ini_{\overrightarrow w}(f)\right)^o$ or $\ini_{\overrightarrow w}(f)=y\left(\ini_{\overrightarrow w}(f)\right)^o$, where $\left(\ini_{\overrightarrow w}(f)\right)^o$ is convenient.\\

\noindent This case follows from  Hensel's Lemma (Lemma \ref{Hensel}), Case 1 and Case 3.\\

\noindent {Case 5:} $\ini_{\overrightarrow w}(f)=xy\left(\ini_{\overrightarrow w}(f)\right)^o$, where $\left(\ini_{\overrightarrow w}(f)\right)^o$ is convenient.\\

\noindent This case follows from  Hensel's Lemma (Lemma \ref{Hensel}), Case 2 and Case 3.\\
\end{proof}

\begin{teorema}
\label{tt}
Suppose that $\ini_{\overrightarrow w}(f)$ has no multiple factors. Then $f$ is tame if and only if $f$ is a semi-quasihomogeneous singularity with respect to $\overrightarrow w$.
\end{teorema}
\noindent \begin{proof}
We have $\overline{\mu}(f)=\left(\frac{\ord_{\overrightarrow w}(f)}{n}-1\right)\left(\frac{\ord_{\overrightarrow w}(f)}{m}-1\right)$ by Proposition \ref{qh}. On the other hand, by Lemma \ref{A2}, we get that
${\mu}(f)=\left(\frac{\ord_{\overrightarrow w}(f)}{n}-1\right)\left(\frac{\ord_{\overrightarrow w}(f)}{m}-1\right)$ if and only if the system of equations 
\[\left\{ \begin{array}{l}
\frac{\partial}{\partial x} \ini_{\overrightarrow w}(f)=0,\\
\\
 \frac{\partial }{\partial y}\ini_{\overrightarrow w}(f)=0 
 \end{array}
 \right.
 \]
 
\noindent has the only solution $(x,y)=(0,0)$. The theorem  follows from Proposition \ref{pp:222}\emph{(iv)}.
\end{proof}

\begin{ejemplo}
Let $f(x,y)=x^m+y^n+\sum_{\alpha n+\beta m>nm} c_{\alpha \, \beta}x^{\alpha}y^{\beta}$ and let $d=\gcd(m,n)$. Then $\ini_{\overrightarrow w}(f)=x^m+y^n$ has no multiple factors if and only if $d\not\equiv 0$ (mod $p$). If $d\not\equiv 0$ (mod $p$) then $f$ is tame if and only if $m\not\equiv 0$ (mod $p$) and $n\not\equiv 0$ (mod $p$).
\end{ejemplo}

\begin{coro}
\label{cccc}
The semi-quasihomogeneous singularities are tame.
\end{coro}

\noindent Corollary \ref{cccc} is a particular case of the following

\begin{teorema} {\rm (Boubakri, Greuel, Markwig \cite[Theorem 9]{BGreuel}).}
The planar Newton non-degenerate singularities are tame.
\end{teorema}

\section{Teissier's lemma in characteristic $p\geq 0$}
\label{Teissier}

\noindent  The intersection  theoretical approach to the Milnor number in characteristic zero \cite{Cassou-P} is based on a lemma due to Teissier who proved a more general result (the case of hypersurfaces) in \cite[Chapter II, Proposition 1.2]{Teissier}. A general formula on isolated complete intersection singularity is due to Greuel \cite{G} and L\^e \cite{Le}. In this section we study Teissier's Lemma in arbitrary characteristic $p\geq 0$.

\medskip

\noindent Let $f\in \bK[[x,y]]$ be a reduced power series and $l\in \bK[[x,y]]$ be a regular parameter. Assume that $l$ does not divide $f$ and consider the polar ${\cal P}_l(f)=\frac{\partial f}{\partial x}\frac{\partial l}{\partial y}-
\frac{\partial f}{\partial y}\frac{\partial l}{\partial x}$ of $f$ with respect to $l$.
In this section we assume, without loss of generality,  that $\ord(l(0,y))=1$.

\begin{lema}
\label{Le:1}
Let $f\in \bK[[x,y]]$ be a reduced power series and $l\in \bK[[x,y]]$ be a regular parameter. Then $i_0\left(l,{\cal P}_l(f)\right)\geq i_0(f,l)-1$ with equality if and only if $i_0(f,l)\not\equiv 0$ \hbox{\rm(}$\mo p$\hbox{\rm)}.
\end{lema} 
\noindent \begin{proof}
Recall that $\ord(l(0,y))=1$. Let $\phi(t)=(\phi_1(t),\phi_2(t))$ be a good parametrization of the curve $l(x,y)=0$ (see \cite[Section 2]{Ploski2013}. In particular $0=l(\phi(t))$ so $\frac{d}{dt}l(\phi(t))=0$. On the other hand we have $\ord(\phi_1(t))=i_0(x,l)=\ord(l(0,y))=1$ and $ \phi'_1(0)\neq 0$. Differentiating $f(\phi(t))$ and  $l(\phi(t))$ we get

\begin{equation}
\label{mmm:1}
\frac{d}{dt}f(\phi(t))=\frac{\partial f}{\partial x}(\phi(t))\phi'_1(t)+\frac{\partial f}{\partial y}(\phi(t))\phi'_2(t)\\
\end{equation}

\noindent and
\begin{equation}
\label{mmm:2}
0=\frac{d}{dt}l(\phi(t))=\frac{\partial l}{\partial x}(\phi(t))\phi'_1(t)+\frac{\partial l}{\partial y}(\phi(t))\phi'_2(t).\\
\end{equation}

\noindent From (\ref{mmm:2})  we have  $\frac{\partial l}{\partial x}(\phi(t))\phi'_1(t)=-\frac{\partial l}{\partial y}(\phi(t))\phi'_2(t)$ and by (\ref{mmm:1}) and the definition of ${\cal P}_l(f)$  we get

\[
{\cal P}_l(f)(\phi(t))\phi'_1(t)=\frac{d}{dt}f(\phi(t))\frac{\partial l}{\partial y}(\phi(t)).
\]

\noindent  Since $\phi'_1(t)$ and $\frac{\partial l}{\partial y}(\phi(t))$ are units in $\bK[[t]]$ we have

\[
\ord({\cal P}_l(f)(\phi(t)))=\ord\left(\frac{d}{dt}f(\phi(t))\right)\geq \ord(f(\phi(t)))-1,
\]

\noindent with equality if and only if $\ord(f(\phi(t)))\not\equiv 0$ ($\mo p$). Now the lemma follows from the formula $i_0(h,l)=\ord(h(\phi(t)))$ which holds for every power series $h\in \bK[[x,y]]$.
\end{proof}

\begin{coro}
\label{c:111}
Suppose that $i_0(f,l)=\ord(f)\not\equiv 0$ \hbox{\rm(}$\mo p$\hbox{\rm)} for a regular parameter $l\in \bK[[x,y]]$. Then
\begin{enumerate}
\item[(a)] $i_0(l,{\cal P}_l(f))=\ord(f)-1$,
\item[(b)] $\ord({\cal P}_l(f))=\ord(f)-1$,
\item[(c)] if $h$ is an irreducible factor of ${\cal P}_l(f)$ then $i_0(l,h)=\ord(h)$.
\end{enumerate}
\end{coro}
\noindent \begin{proof} Property $(a)$  follows immediately from Lemma \ref{Le:1}. To check $(b)$ observe that we get $\ord({\cal P}_l(f))=\ord({\cal P}_l(f))\cdot \ord(l)\leq i_0(l, {\cal P}_l(f))=\ord(f)-1$, where the last equality follows from  $(a)$. The inequality
$\ord({\cal P}_l(f))\geq \ord(f)-1$  is obvious.

\medskip

\noindent Let ${\cal P}_l(f)=\prod_{i=1}^s h_i$, where $h_i$ is irreducible. From $(a)$ and $(b)$ we get 
\[
0=i_0(l,{\cal P}_l(f))-\ord({\cal P}_l(f))=\sum_{i=1}^s\left(i_0(l,h_i)-
\ord(h_i)\right).
\]

\noindent Since $i_0(l,h_i)\geq \ord(h_i)$  we have $i_0(l,h_i)=\ord(h_i)$ for $i=1,\ldots,s$ which proves $(c)$.
\end{proof}

\begin{prop}{\rm(Teissier's Lemma in characteristic $p$).}
\label{Teissier2}
Let $f\in \bK[[x,y]]$ be a reduced power series. Suppose that
\begin{enumerate}
\item[(i)] $i_0(f,l)\not\equiv 0$ \hbox{\rm(}$\mo p$\hbox{\rm)}, 
\item[(ii)] for any irreducible factor $h$ of $ {\cal P}_l(f)$ we get $i_0(l,h) \not\equiv 0$ \hbox{\rm(}$\mo p$\hbox{\rm)}. 
\end{enumerate}
\noindent Then \[i_0(f,{\cal P}_l(f))\leq \mu(f)+i_0(f,l)-1\] with equality if and only if
\begin{enumerate}
\item[(iii)] for any irreducible factor $h$ of ${\cal P}_l(f)$ we get $i_0(f,h) \not\equiv 0$ \hbox{\rm(}$\mo p$\hbox{\rm)}. 
\end{enumerate}
\end{prop}
\noindent \begin{proof} Fix an irreducible factor $h$ of ${\cal P}_l(f)$ and let 
$\psi(t)=(\psi_1(t),\psi_2(t))$ be a good parametrization of the curve $h(x,y)=0$. Then $\ord(l(\psi(t)))=i_0(l,h)\not\equiv 0$ \hbox{\rm(}$\mo p$\hbox{\rm)} by $(ii)$
and $\ord\left(\frac{d}{dt}l(\psi(t))\right)=\ord(l(\psi(t)))-1$. Differentiating $f(\psi(t))$ and $l(\psi(t))$ we get

\begin{equation}
\label{mm:1}
\frac{d}{dt}f(\psi(t))=\frac{\partial f}{\partial x}(\psi(t))\psi'_1(t)+\frac{\partial f}{\partial y}(\psi(t))\psi'_2(t),
\end{equation}

\noindent and

\begin{equation}
\label{mm:2}
\frac{d}{dt}l(\psi(t))=\frac{\partial l}{\partial x}(\psi(t))\psi'_1(t)+\frac{\partial l}{\partial y}(\psi(t))\psi'_2(t).
\end{equation}

\noindent Since  ${\cal P}_l(f)(\psi(t))=0$, it follows from (\ref{mm:1}) and (\ref{mm:2}) that

\begin{equation}
\label{mm:3}
\frac{d}{dt}f(\psi(t))\frac{\partial l}{\partial y}(\psi(t))=\frac{d}{dt}l(\psi(t))\frac{\partial f}{\partial y}(\psi(t)).
\end{equation}

\noindent Since $\frac{\partial l}{\partial y}(\psi(t))$ is a unit in $\bK[[t]]$, taking orders  in (\ref{mm:3}) we have

\begin{eqnarray*}
\ord(f(\psi(t)))-1 & \leq &\ord\left(\frac{d}{dt}f(\psi(t))\right)=\ord \left(\frac{d}{dt}l(\psi(t))\right)+\ord \left(\frac{\partial f}{\partial y}(\psi(t))\right)\\
&=&\ord(l(\psi(t)))-1+\ord\left(\frac{\partial f}{\partial y}(\psi(t))\right),
\end{eqnarray*}

\noindent where the last equality follows from $\ord(l(\psi(t)))\not\equiv 0$ \hbox{\rm(}$\mo p$\hbox{\rm)}.\\ Hence $i_0(f,h)\leq i_0(l,h)+i_0\left(\frac{\partial f}{\partial y},h\right).$

\noindent Summing up over all $h$ counted with multiplicities as factors of
${\cal P}_l(f)$ we obtain

\begin{equation}
\label{qqq}i_0(f,{\cal P}_l(f))\leq i_0(l,{\cal P}_l(f))+i_0\left(\frac{\partial f}{\partial y},{\cal P}_l(f)\right).
\end{equation} 

\noindent By Lemma \ref{Le:1} and assumption $(i)$ we have
$i_0\left(l,{\cal P}_l(f)\right)=i_0(f,l)-1.$ Moreover $i_0\left(\frac{\partial f}{\partial y},{\cal P}_l(f)\right)=\mu(f)$ since $\ord(l(0,y))=1$ and we get from the equality (\ref{qqq})
 \[i_0(f,{\cal P}_l(f))\leq \mu(f)+i_0(f,l)-1.\]
\medskip

\noindent The equality holds if and only if $i_0(f,h)=i_0(l,h)+i_0\left(\frac{\partial f}{\partial y},h\right)$ for every $h$, which is equivalent to the condition $i_0(f,h) \not\equiv 0$ \hbox{\rm(}$\mo p$\hbox{\rm)}, since $i_0(f,h) \not\equiv 0$ \hbox{\rm(}$\mo p$\hbox{\rm)} if and only if $\ord\left(\frac{d}{dt}f(\psi(t))\right)=\ord(f(\psi(t)))-1$.
\end{proof}

\begin{coro}{\rm (Teissier \cite[Chapter II, Proposition 1.2]{Teissier}).}
\label{car=0}
If $\ch \bK=0$ then $$i_0(f,{\cal P}_l(f))= \mu(f)+i_0(f,l)-1.$$
\end{coro}

\begin{coro} Suppose that $p=\ch \bK> \ord(f)$ and let $i_0(f,l)=\ord (f)$. Then
 \[i_0({\cal P}_l(f),f)\leq \mu(f)+i_0(f,l)-1.\]
 \noindent The equality holds if and only if for any irreducible factor $h$ of 
 ${\cal P}_l(f)$  we get $i_0(f,h) \not\equiv 0$ \hbox{\rm(}$\mo p$\hbox{\rm)}. 
\end{coro}
\noindent \begin{proof}  If $\ord(f)<p$ then $i_0(f,l)=\ord(f)\not\equiv 0$ \hbox{\rm(}$\mo p$\hbox{\rm)} and by Corollary \ref{c:111} for any irreducible factor $h$ of ${\cal P}_l(f)$ we get 
\[
i_0(l,h)=\ord(h) \leq \ord({\cal P}_l(f))=\ord(f) -1<p.
\]

\noindent Hence $i_0(l,h)\not\equiv 0$ \hbox{\rm(}$\mo p$\hbox{\rm)} and the corollary follows from Proposition \ref{Teissier2}.
\end{proof}

\begin{ejemplo}
Let $f=x^{p+2}+y^{p+1}+x^{p+1}y$, where $p=\ch K>2$. Take $l=y$. Then $i_0(f,l)=p+2\not\equiv 0$ \hbox{\rm(}$\mo p$\hbox{\rm)}, ${\cal P}_l(f)=\frac{\partial f}{\partial x}=x^p(2x+y)$ and the irreducible factors of ${\cal P}_l(f)$ are $h_1=x$ and $h_2=2x+y$. Clearly $i_0(l,h_1)=i_0(l,h_2)=1\not\equiv 0$  \hbox{\rm(}$\mo p$\hbox{\rm)}. Moreover $i_0(f,h_1)=i_0(f,h_2)=p+1$ and all assumptions of Proposition \ref{Teissier2} are satisfied.

\medskip

\noindent Hence $i_0(f,{\cal P}_l(f))=\mu(f)+i_0(f,l)-1$ and $\mu(f)=i_0(f,{\cal P}_l(f))-i_0(f,l)+1=p(p+1)$. Note that $l=0$ is a curve of maximal contact with $f=0$. Let $l_1=x$. Then $i_0(f,l_1)=\ord(f)=p+1$, ${\cal P}_{l_1}(f)=-(y^p+x^{p+1})$ and $h=y^p+x^{p+1}$ is the only irreducible factor of the polar ${\cal P}_{l_1}(f)$. Since $i_0(l_1,h)=p$, the condition $(ii)$ of Proposition \ref{Teissier2} is not satisfied. However, $i_0(f,{\cal P}_{l_1}(f))=\mu(f)+i_0(f,l_1)-1,$ which we check directly.
\end{ejemplo}

\section{Tame singularities}

\noindent Assume that $f$ is a plane curve singularity.

\begin{prop}
\label{PPP}
Let $f=f_1\cdots f_r \in \bK[[x,y]]$ be a reduced power series, where $f_i$ is irreducible for $i=1,\ldots, r$. Suppose that there exists a regular parameter $l$ such that $i_0(f_i,l)\not\equiv 0$  \hbox{\rm(}$\mo p$\hbox{\rm)} for $i=1,\ldots, r$. Then $f$ is tame if and only if Teissier's lemma holds, that is if $i_0(f,{\cal P}_l(f))=\mu(f)+i_0(f,l)-1$.
\end{prop}
\noindent \begin{proof}
By Proposition \ref{pp:222} $(iii)$ we have that $i_0(f,{\cal P}_l(f))=\overline{\mu}(f)+i_0(f,l)-1$. Thus $i_0(f,{\cal P}_l(f))=\mu(f)+i_0(f,l)-1$ if and only if $\mu(f)=\overline{\mu}(f)$. We finish the proof using Proposition \ref{pp:222} $(iv)$.
\end{proof}

\begin{prop}{\rm (Milnor \cite{Milnor}, Risler \cite{Risler}).}
If $\ch \bK=0$ then any plane singularity is tame.
\end{prop}
\noindent \begin{proof}
Teissier's Lemma holds by Corollary \ref{car=0} . Use Proposition \ref{PPP}.
\end{proof}

\begin{prop} 
\label{QQQ} Let $p=\ch \bK>0$. Suppose that $p>\ord(f)$. Let $l$ be a regular parameter such that $i_0(f,l)=\ord(f)$. Then $f$ is tame if and only if for any irreducible factor $h$ of ${\cal P}_l(f)$ we get $i_0(f,h)\not\equiv 0$  \hbox{\rm(}$\mo p$\hbox{\rm)}.
\end{prop}
\noindent \begin{proof}  Take a regular parameter $l$  such that $i_0(f,l)=\ord(f)$. By hypothesis we get $i_0(f,l)<p$ so $i_0(f,l)\not\equiv 0$  \hbox{\rm(}$\mo p$\hbox{\rm)}. By Corollary \ref{c:111} the assumption $(ii)$ of Proposition \ref{Teissier2} is satisfied.

\medskip

\noindent Hence $i_0(f, {\cal P}_l(f))\leq \mu(f)+i_0(f,l)-1$ with equality if and only if $i_0(f,h)\not\equiv 0$  \hbox{\rm(}$\mo p$\hbox{\rm)} for any irreducible factor $h$ of ${\cal P}_l(f)$. Use Proposition \ref{PPP}.
\end{proof}

\begin{prop}{\rm (Nguyen \cite{Nguyen}).}
\label{Ngu}
Let $p=\ch \bK>0$. Suppose that there exists a regular parameter $l$ such that $i_0(f,l)=\ord(f)$ and $i_0(f,{\cal P}_l(f))<p$. Then $f$ is tame.
\end{prop}
\noindent \begin{proof}
We have $p>i_0(f,{\cal P}_l(f))\geq \ord(f) \cdot \ord({\cal P}_l(f)) $. Hence $p>\ord(f)$ and we may apply Proposition \ref{QQQ}. Since $i_0(f,{\cal P}_l(f))<p$ for any irreducible factor $h$ of ${\cal P}_l(f)$ we have that $i_0(f,h)<p$ and obviously $i_0(f,h)\not\equiv 0$  \hbox{\rm(}$\mo p$\hbox{\rm)}. The proposition follows from Proposition \ref{QQQ}.
\end{proof}

\begin{teorema}{\rm(Nguyen \cite{Nguyen}).}
If $p>\mu(f)+\ord(f)-1$ then $f$ is tame.
\end{teorema}
\noindent \begin{proof}
Since $f$ is a singularity we get $\mu(f)>0$ and by hypothesis the characteristic of the field verifies $p>\mu(f)-1+\ord(f)\geq \ord(f)$. By the first part of the proof of Proposition \ref{QQQ} we have $i_0(f, {\cal P}_l(f))\leq \mu(f)+\ord(f)-1$, where $l$ is a regular parameter such that $i_0(f,l)=\ord(f)$. Hence $i_0(f, {\cal P}_l(f))<p$ and the theorem follows from Proposition \ref{Ngu}.
\end{proof}

\section{The Milnor number of plane irreducible singularities}

\noindent Let $f	\in \bK[[x,y]]$ be an irreducible power series of order $n=\ord(f)$ and let  $\Gamma(f)$ be the semigroup associated with  $f=0$. 
\medskip

\noindent Let $\overline{\beta_0},\ldots,\overline{\beta_g}$ be the minimal
sequence of generators of $\Gamma(f)$ defined by the conditions
\begin{itemize}
\item $\overline{\beta_0}=\min (\Gamma(f)\backslash\{0\})=\ord(f)=n$,
\item $\overline{\beta_k}=\min (\Gamma(f)\backslash \bN \overline{\beta_0}+\cdots+\bN \overline{\beta_{k-1}})$ for $k\in\{1,\ldots,g\},$
\item $\Gamma(f)=\bN \overline{\beta_0}+\cdots+\bN \overline{\beta_{g}}$.
\end{itemize}

\noindent Let $e_k=\gcd(\overline{\beta_0},\ldots,\overline{\beta_k})$ for $k\in\{1,\ldots,g\}.$ Then $n=e_0>e_1>\cdots e_{g-1}>e_g=1$. Let $n_k=e_{k-1}/e_k$ for 
$k\in\{1,\ldots,g\}.$ We have $n_k>1$ for $k\in\{1,\ldots,g\}$ and $n=n_1\cdots n_g$. Let $n^*=\max(n_1,\ldots,n_g)$. Then $n^*\leq n$ with equality if and only if $g=1$.

\medskip

\noindent The following theorem is a sharpened version of the main result of \cite{GB-P2016}.

\begin{teorema}
\label{main}
Let $f\in \bK[[x,y]]$ be an irreducible power series of order $n>1$ and let $\overline{\beta_0},\ldots,\overline{\beta_g}$ be the minimal
system of generators of $\Gamma(f)$. Suppose that $p=\ch \bK>n^*$. Then the following two conditions are equivalent:
\begin{enumerate}
\item[(i)]$\overline{\beta_k}\not\equiv 0$ \hbox{\rm(}$\mo p$\hbox{\rm)} for $k\in\{1,\ldots,g\},$
\item [(ii)]$f$ is tame.
\end{enumerate}
\end{teorema}

\noindent In \cite{GB-P2016} the equivalence of $(i)$ and $(ii)$ is proved under the assumption that $p>n$.

\medskip

\noindent If $f\in \bK[[x,y]]$ is an irreducible power series then we get $\ord(f(x,0))=\ord(f)$ or $\ord(f(0,y))=\ord(f)$. In the sequel we assume that $\ord(f(0,y))=\ord(f)=n$. The proof of Theorem \ref{main} is based on Merle's factorization theorem:

\begin{teorema}{\rm (Merle \cite{Merle}, Garc\'{\i}a Barroso-P\l oski \cite{GB-P2016}).}\\
\label{decomposition}
Suppose that $\ord(f(0,y))=\ord(f)=n\not \equiv 0$ \hbox{\rm(}$\mo p$\hbox{\rm)}. Then $\frac{\partial f}{\partial y}=h_1\cdots h_g$ in $\bK[[x,y]],$ where

\begin{enumerate}
\item[(a)] $\ord(h_k)=\frac{n}{e_k}-\frac{n}{e_{k-1}}$ for $k\in \{1,\ldots,g\}$.
\item[(b)] If $h\in \bK[[x,y]]$ is an irreducible factor of $h_k$, $k\in \{1,\ldots,g\}$, then 
\begin{enumerate}
\item[(b1)] $\frac{i_0(f,h)}{\ord(h)}=\frac{e_{k-1}\overline{\beta_k}}{n},$
\noindent and 
\item[(b2)] $\ord(h)\equiv 0$  $\left(mod \;\frac{n}{e_{k-1}}\right)$.
\end{enumerate}
\end{enumerate}
\end{teorema}

\begin{lema} 
\label{mu}
Suppose that $p>n^*$. Then $i_0\left(f, \frac{\partial f}{\partial y}\right)\leq \mu(f)+\ord (f)-1$ with equality if and only if $\overline{\beta_k}\not\equiv 0$ \hbox{\rm(}$\mo p$\hbox{\rm)} for $k\in \{0, \ldots, g\}$.
\end{lema}
\noindent \begin{proof} Obviously $n_k\not\equiv 0$ \hbox{\rm(}$\mo p$\hbox{\rm)} for $k=1,\ldots,g$ and $n=n_1\cdots n_g \not\equiv 0$ \hbox{\rm(}$\mo p$\hbox{\rm)}. Let $h$ be an irreducible factor of  $\frac{\partial f}{\partial y}$. Then, by Corollary \ref{c:111}$(c)$ $i_0(h,x)=\ord(h)$.  By Theorem \ref{decomposition} {\em (b2)}  $\ord(h)=m_k\frac{n}{e_{k-1}},$ for an index $k\in\{1, \ldots, g\}$,  where $m_k\geq 1$ is an integer. 
Hence $m_k\frac{n}{e_{k-1}}=\ord(h)\leq \ord(h_k)=\frac{n}{e_{k-1}}(n_k-1)$ and $m_k\leq n_k-1<n_k<p$, which implies $m_k\not\equiv 0$ \hbox{\rm(}$\mo p$\hbox{\rm)} and $\ord(h) \not\equiv 0$ \hbox{\rm(}$\mo p$\hbox{\rm)}.
By Proposition \ref{Teissier2} we get $i_0\left(f,\frac{\partial f}{\partial y}\right)\leq \mu(f)+\ord(f)-1$.
 By Theorem \ref{decomposition} {\em (b1)} we have
 the equalities $i_0(f,h)=
\left(\frac{e_{k-1}\overline{\beta_k}}{n}\right)\ord(h)=m_k \overline{\beta_k}$ and we get $i_0(f,h)\not\equiv 0$ \hbox{\rm(}$\mo p$\hbox{\rm)} if and only if $\overline{\beta_k}\not\equiv 0$ \hbox{\rm(}$\mo p$\hbox{\rm)}, which proves the second part of Lemma \ref{mu}. 
\end{proof}

\medskip

\noindent {\bf Proof of Theorem \ref{main}}
Use Lemma \ref{mu} and Proposition \ref{PPP}. 
{\rule{0.5em}{0.5em}}

\begin{ejemplo}
Let  $f(x,y)=(y^2+x^3)^2+x^5y$. Then $f$ is irreducible and its semigroup is $\Gamma(f)=4\bN+6\bN+13\bN$. Here $e_0=4$, $e_1=2$, $e_2=1$ and $n_1=n_2=2$. Hence $n^*=2$. \\Let $p>n^*=2$. If $p=\ch \bK\neq 3, 13$ then $f$ is tame. On the other hand if $p=2$ then $\mu(f)=+\infty$ since $x$ is a common factor of $\frac{\partial f}{\partial y}$ and $\frac{\partial f}{\partial x}$. Hence $f$ is tame if and only if $p\neq 2,3,13$. Note that for any $l$ with $\ord(l)=1$ we have $i_0(f,l)\equiv 0$  \hbox{\rm(}$\mo 2$\hbox{\rm)}.
\end{ejemplo}

\begin{prop}
If $\Gamma(f)=\overline{\beta_0}\bN+\overline{\beta_1}\bN$ then $f$ is tame if and only if 
$\overline{\beta_0}\not\equiv 0$  \hbox{\rm(}$\mo p$\hbox{\rm)} and $\overline{\beta_1}\not\equiv 0$  \hbox{\rm(}$\mo p$\hbox{\rm)}.
\end{prop}

\noindent \begin{proof} 
Let $\overrightarrow w=(\overline{\beta_0},\overline{\beta_1})$. There exists a system of coordinates $x,y$ such that we can write $f=y^{\overline{\beta_0}}+x^{\overline{\beta_1}}+\hbox{\rm terms of weight greater than } \overline{\beta_0} \,\overline{\beta_1}$. The proposition follows from Theorem \ref{tt} (see also \cite[Example 2]{GB-P2016}). 
\end{proof}

\medskip

\noindent In \cite{Hefez}  the authors proved, without any restriction on $p=\ch \bK$, the following profound result:

\begin{teorema}{\rm (Hefez, Rodrigues, Salom\~ao \cite{Hefez}, \cite{Hefez2}).}
\label{H}
Let $\Gamma(f)=\overline{\beta_0}\bN+\cdots+\overline{\beta_g}\bN$. If $\overline \beta_k\not\equiv 0$ \hbox{\rm(}$\mo p$\hbox{\rm)} for $k=0,\ldots, g$ then $f$ is tame.
\end{teorema}

\noindent The question as to whether the converse of Theorem \ref{H} is true remains open.

\section{Appendix}

\noindent Let $\overrightarrow w=(n,m)\in (\bN_+)^2$  be a weight.

\begin{lema}
\label{A1}
Let $f,g\in \bK[[x,y]]$ be power series without constant term. Then 
\[
i_0(f,g)\geq \frac{\left(\ord_{\overrightarrow w}(f) \right) \left(ord_{\overrightarrow w}(g) \right)}{mn},\]
\noindent with equality if and only if the system of equations 
\[
\left\{\begin{array}{l}
\ini_{\overrightarrow w}(f)=0,\\
\\
\ini_{\overrightarrow w}(g)=0
\end{array}
\right.
\]
\noindent has the only solution $(x,y)=(0,0)$.
\end{lema}
\noindent \begin{proof} 
By a basic property of the intersection multiplicity (see for example \cite[Proposition 3.8 (v)]{Ploski2013}) we have that for any nonzero  power series $\tilde f, \tilde g$ 
\begin{equation}
\label{imin}
i_0(\tilde f, \tilde g)\geq \ord(\tilde f) \ord(\tilde g),
\end{equation}

\noindent  with equality if and only if the system of equations $\ini(\tilde f)=0$, $\ini(\tilde g)=0$ has the only solution $(0,0)$. Consider the  power series $\tilde f(u,v)=f(u^n,v^m)$ and $\tilde g(u,v)=g(u^n,v^m)$. Then  $i_0(\tilde f, \tilde g)=i_0(f,g)i_0(u^n,v^m)=i_0(f,g)nm$, 
$\ord(\tilde f)=\ord_{\overrightarrow w}(f)$, $\ord(\tilde g)=\ord_{\overrightarrow w}(g)$ and the lemma follows from (\ref{imin}).
 \end{proof}

\begin{lema}
\label{A2}
Let $f\in \bK[[x,y]]$ be a non-zero power series. Then 
\[
i_0\left(\frac{\partial f}{\partial x}, \frac{\partial f}{\partial y}\right)\geq \left(\frac{\ord_{\overrightarrow w}(f)}{n}-1\right)\left(\frac{\ord_{\overrightarrow w}(f)}{m}-1\right)
\]
with equality  if and only if $f$ is a semi-quasihomogeneous singularity with respect  to ${\overrightarrow w}$. 
 \end{lema}
 
 \noindent \begin{proof} The following two properties are useful:
 \begin{equation}
 \label{eeeqqq1}
 \ord_{\overrightarrow w} \left(\frac{\partial f}{\partial x}\right)\geq \ord_{\overrightarrow w}(f)-n \;\;\hbox {\rm with equality if and only if } \; \frac{\partial }{\partial x} \ini_{\overrightarrow w}(f)\neq 0,
 \end{equation}
 
  \begin{equation}
 \label{eeeqqq2}
\hbox{\rm if} \; \frac{\partial }{\partial x} \ini_{\overrightarrow w}(f)\neq 0 \;\hbox{\rm then } \;
 \ini_{\overrightarrow w} \left(\frac{\partial f}{\partial x}\right)=\frac{\partial }{\partial x} \ini_{\overrightarrow w}(f).
 \end{equation}
 
 \noindent By the first part of Lemma \ref{A1} and Property (\ref{eeeqqq1}) we get
 
 \begin{eqnarray*}
 i_0\left(\frac{\partial f}{\partial x},\frac{\partial f}{\partial y}\right)&\geq &
  \frac{\left(\ord_{\overrightarrow w}\left( \frac{\partial f}{\partial x}\right) \right) \left(ord_{\overrightarrow w} \left(\frac{\partial f}{\partial y} \right)\right)}{nm}\geq  \frac{\left(\ord_{\overrightarrow w}(f)-n \right) \left(ord_{\overrightarrow w}(f)-m \right)}{nm}\\
  &=&\left(\frac{\ord_{\overrightarrow w}(f)}{n}-1\right)\left(\frac{\ord_{\overrightarrow w}(f)}{m}-1\right).
 \end{eqnarray*}
 
 \noindent Using the second part of Lemma \ref{A1} and Properties (\ref{eeeqqq1}) and (\ref{eeeqqq2}) we check that $ i_0\left(\frac{\partial f}{\partial x},\frac{\partial f}{\partial y}\right)=
 \left(\frac{\ord_{\overrightarrow w}(f)}{n}-1\right)\left(\frac{\ord_{\overrightarrow w}(f)}{m}-1\right)$ if and only if $f$ is a semi-quasihomogeneous singularity with respect  to ${\overrightarrow w}$. 
 \end{proof}

 \begin{lema}{\rm (Hensel's Lemma \cite[Theorem 16.6]{Kunz}).}
 \label{Hensel}
 Suppose that $\ini_{\overrightarrow w}(f)=\psi_1\cdots \psi_s$ with pairwise coprime $\psi_i$. Then $f=g_1\cdots g_s\in \bK[[x,y]]$ with $\ini_{\overrightarrow w} (g_i)=\psi_i$ for $i=1,\ldots,s$.
 \end{lema}

\medskip
\noindent
{\small Evelia Rosa Garc\'{\i}a Barroso\\
Departamento de Matem\'aticas, Estad\'{\i}stica e I.O. \\
Secci\'on de Matem\'aticas, Universidad de La Laguna\\
Apartado de Correos 456\\
38200 La Laguna, Tenerife, Espa\~na\\
e-mail: ergarcia@ull.es}

\medskip

\noindent {\small Arkadiusz P\l oski\\
Department of Mathematics and Physics\\
Kielce University of Technology\\
Al. 1000 L PP7\\
25-314 Kielce, Poland\\
e-mail: matap@tu.kielce.pl}

\end{document}